\documentclass[a4,12pt]{amsart}
\usepackage{amssymb,amstext,amsmath,amscd,amsthm,amsfonts,enumerate,graphicx,latexsym}
\usepackage[usenames]{color}
\usepackage{pstricks,comment}
\newcommand{\new}[1]{{#1}}
\newtheorem{thm}{Theorem}[section]
\newtheorem{cor}[thm]{Corollary}
\newtheorem{lem}[thm]{Lemma}
\newtheorem{prop}[thm]{Proposition}
\newtheorem*{thm*}{Theorem}
\newtheorem*{cor*}{Corollary}
\theoremstyle{definition}
\newtheorem{dfn}[thm]{Definition}
\newtheorem{rem}[thm]{Remark}

\newtheorem{conv}[thm]{Convention}

\newtheorem*{claim*}{Claim}

\newtheorem{nota}[thm]{Notation}

\numberwithin{equation}{thm}
\def\CM{\operatorname{\mathsf{CM}}}
\def\mod{\operatorname{\mathsf{mod}}}
\def\add{\operatorname{\mathsf{add}}}
\def\res{\operatorname{\mathsf{res}}}
\def\proj{\operatorname{\mathsf{proj}}}

\def\Tor{\operatorname{\mathsf{Tor}}}
\def\Ext{\operatorname{\mathsf{Ext}}}
\def\Ann{\operatorname{\mathsf{Ann}}}
\def\lCM{\operatorname{\underline{\mathsf{CM}}}}
\def\NF{\operatorname{\mathsf{NF}}}
\def\V{\operatorname{\mathsf{V}}}
\def\Supp{\operatorname{\mathsf{Supp}}}
\def\Sing{\operatorname{\mathsf{Sing}}}
\def\Spec{\operatorname{\mathsf{Spec}}}
\def\fl{\operatorname{\mathsf{fl}}}

\def\lSupp{\operatorname{\underline{\mathsf{Supp}}}}
\def\Db{\operatorname{\mathsf{D}^b}}
\def\perf{\operatorname{\mathsf{perf}}}
\def\Hom{\operatorname{\mathsf{Hom}}}
\def\lhom{\operatorname{\underline{\mathsf{Hom}}}}

\def\Ker{\operatorname{\mathsf{Ker}}}
\def\Ob{\operatorname{\mathsf{Ob}}}

\def\syz{\mathsf{\Omega}}
\def\sus{\mathsf{\Sigma}}

\def\depth{\operatorname{\mathsf{depth}}}
\def\radi{\operatorname{\mathsf{radius}}}
\def\e{\operatorname{\mathsf{e}}}
\def\edim{\operatorname{\mathsf{edim}}}
\def\dim{\operatorname{\mathsf{dim}}}

\def\A{\mathcal{A}}
\def\X{\mathcal{X}}
\def\Y{\mathcal{Y}}
\def\Z{\mathcal{Z}}
\def\C{\mathcal{C}}
\def\T{\mathcal{T}}
\def\P{\mathcal{P}}

\def\CC{\mathbb{C}}
\def\ZZ{\mathbb{Z}}
\def\N{\mathbb{N}}

\def\m{\mathfrak{m}}
\def\p{\mathfrak{p}}
\def\q{\mathfrak{q}}
\def\NN{\mathfrak{N}}

\begin{document}
\setlength{\baselineskip}{15pt}
\title{The dimension of a subcategory of modules}
\author{Hailong Dao}
\address{Department of Mathematics, University of Kansas, Lawrence, KS 66045-7523, USA}
\email{hdao@math.ku.edu}
\urladdr{http://www.math.ku.edu/~hdao/}
\author{Ryo Takahashi}
\address{Graduate School of Mathematics, Nagoya University, Furocho, Chikusaku, Nagoya 464-8602, Japan/Department of Mathematics, University of Nebraska, Lincoln, NE 68588-0130, USA}
\email{takahashi@math.nagoya-u.ac.jp}
\urladdr{http://www.math.nagoya-u.ac.jp/~takahashi/}
\thanks{The first author was partially supported by NSF grants DMS 0834050 and DMS 1104017. The second author was partially supported by JSPS Grant-in-Aid for Young Scientists (B) 22740008 and by JSPS Postdoctoral Fellowships for Research Abroad}
\begin{abstract}
Let $R$ be a commutative noetherian local ring.
As an analogue of the notion of the dimension of a triangulated category defined by Rouquier, the notion of the dimension of a subcategory of finitely generated $R$-modules is introduced in this paper. We found  evidence that certain categories over  nice singularities have small dimensions. 
When $R$ is Cohen-Macaulay, under a mild assumption it is proved that finiteness of the dimension of the full subcategory consisting of maximal Cohen-Macaulay modules which are locally free on the punctured spectrum is equivalent to saying that $R$ is an isolated singularity.
As an application, the celebrated theorem of Auslander, Huneke, Leuschke and Wiegand is not only recovered but also improved.
The dimensions of stable categories of maximal Cohen-Macaulay modules as triangulated categories are also investigated in the case where $R$ is Gorenstein, and special cases of the recent results of Aihara and Takahashi, and Oppermann and \v{S}\'{t}ov\'{i}\v{c}ek are recovered and improved.
Our key technique involves a careful study of annihilators and supports of $\Tor$, $\Ext$ and $\lhom$ between two subcategories.
\end{abstract}
\maketitle
\section{Introduction}

The notion of the dimension of a triangulated category has been introduced implicitly by Bondal and Van den Bergh \cite{BV} and explicitly by Rouquier \cite{R}.
It is defined as the number of \new{triangles} necessary to build the category from a single object, up to finite direct sum, direct summand and shift.
Rouquier proved that the bounded derived category of coherent sheaves on a separated scheme of finite type over a perfect field has finite dimension.
Finiteness of the dimension of the bounded derived category of finitely generated modules over a complete local ring with perfect coefficient field was recently proved by Aihara and Takahashi \cite{ddc}.

The concept of a thick subcategory of a triangulated category has been introduced by Verdier \cite{V} to develop the theory of localizations of triangulated categories.
Thick subcategories have been studied widely and deeply so far, mainly from the motivation to classify them; see \cite{BCR,BIK,DHS,FP,Ho,HS,N,stcm,To} for instance.
Since a thick subcategory is a triangulated category, its dimension in the sense of Rouquier can be defined.
It turned out by Oppermann and \v{S}\'{t}ov\'{i}\v{c}ek \cite{OS} that over a noetherian algebra (respectively, a projective scheme) all proper thick subcategories of the bounded derived category of finitely generated modules (respectively, coherent sheaves) containing perfect complexes have infinite dimension.

The concept of a resolving subcategory of an abelian category has been introduced by Auslander and Bridger \cite{AB}.
They proved that in the category of finitely generated modules over a noetherian ring the full subcategory consisting of modules of Gorenstein dimension zero is resolving.
A landmark development concerning resolving subcategories was made by Auslander and Reiten \cite{AR} in connection with tilting theory.
Recently, several studies on resolving subcategories have been done by Dao and Takahashi \cite{radius,resreg,res,stcm,arg,crs}.

In this paper, we introduce an analogue of the notion of the dimension of a triangulated category for full subcategories $\X$ of an abelian category with enough projective objects.
To be precise, we define the dimension of $\X$ as the number of extensions necessary to build $\X$ from a single object in $\X$, up to finite direct sum, direct summand and syzygy.
To state our results, let us fix some notation.
Let $R$ be a Cohen-Macaulay local ring.
Denote by $\CM(R)$ the category of maximal Cohen-Macaulay $R$-modules, and by $\CM_0(R)$ the category of maximal Cohen-Macaulay $R$-modules that are locally free on the punctured spectrum.
These two categories are resolving subcategories of the category $\mod R$ of finitely generated $R$-modules.
The stable categories of $\CM(R)$ and $\CM_0(R)$ are denoted by $\lCM(R)$ and $\lCM_0(R)$, respectively.
When $R$ is Gorenstein, $\lCM(R)$ is a triangulated category \cite{B,H}, and $\lCM_0(R)$ is a thick subcategory of $\lCM(R)$.
The main purpose of this paper is to investigate finiteness of the dimensions of resolving subcategories of $\mod R$, and the dimensions of thick subcategories of $\lCM(R)$ in the case where $R$ is Gorenstein.
Our first main result is a characterization of the isolated singularity of $R$ in terms of the dimensions of $\CM_0(R)$ and $\lCM_0(R)$:

\begin{thm}\label{1.1}
Let $R$ be a Cohen-Macaulay local ring with maximal ideal $\m$.
\begin{enumerate}[\rm(1)]
\item
Consider the following four conditions.
\begin{enumerate}[\rm(a)]
\item
The dimension of $\CM_0(R)$ is finite.
\item
The ideal $\bigcap_{i>0,\,M,N\in\CM_0(R)}\Ann_R\Ext_R^i(M,N)$ is $\m$-primary.
\item
The ideal $\bigcap_{i>0,\,M,N\in\CM_0(R)}\Ann_R\Tor_i^R(M,N)$ is $\m$-primary.
\item
The ring $R$ has at most an isolated singularity.
\end{enumerate}
Then, the implications ${\rm(a)}\Leftrightarrow{\rm(b)}\Rightarrow{\rm(c)}\Rightarrow{\rm(d)}$ hold.
The implication ${\rm(d)}\Rightarrow{\rm(a)}$ also holds if $R$ is complete, equicharacteristic and with perfect residue field.
\item
Suppose that $R$ is Gorenstein, and consider the following three conditions.
\begin{enumerate}[\rm(a)]
\item
The dimension of the triangulated category $\lCM_0(R)$ is finite.
\item
The annihilator of the $R$-linear category $\lCM_0(R)$ is $\m$-primary.
\item
The ring $R$ has at most an isolated singularity.
\end{enumerate}
Then the implications ${\rm(a)}\Leftrightarrow{\rm(b)}\Rightarrow{\rm(c)}$ hold, and so does ${\rm(c)}\Rightarrow{\rm(a)}$ if $R$ is complete, equicharacteristic and with perfect residue field.
\end{enumerate}
\end{thm}

The celebrated Auslander-Huneke-Leuschke-Wiegand theorem states that every Cohen-Macaulay local ring of finite Cohen-Macaulay representation type has at most an isolated singularity.
This was proved by Auslander \cite{A} in the case where the ring is complete, by Leuschke and Wiegand \cite{LW} in the case where the ring is excellent, and by Huneke and Leuschke \cite{HL} in the general case.
Our Theorem \ref{1.1} not only deduces this result but also improves it:

\begin{cor}[Improved Auslander-Huneke-Leuschke-Wiegand Theorem]
Let $R$ be a Cohen-Macaulay local ring.
Suppose that there are only finitely many isomorphism classes of indecomposable maximal Cohen-Macaulay $R$-modules which are locally free on the punctured spectrum.
Then $R$ has at most an isolated singularity.
\end{cor}

This result very easily follows from the first assertion of Theorem \ref{1.1}.
Indeed, the assumption of the corollary implies that the dimension of $\CM_0(R)$ is zero, hence is finite. \new{Our methods can be adapted to give effective versions of this result, see \cite{sing}}.

Our Theorem \ref{1.1} also gives rise to finiteness of the dimensions of $\CM(R)$ and $\lCM(R)$ when $R$ has an isolated singularity, the latter of which is a special case of the main result of Aihara and Takahashi \cite{ddc}.

\begin{cor}
Let $R$ be a Cohen-Macaulay excellent local ring with perfect coefficient field.
Suppose that $R$ has at most an isolated singularity.
Then $\CM(R)$ is of finite dimension.
If $R$ is Gorenstein, then $\lCM(R)$ is of finite dimension as a triangulated category.
\end{cor}

Our second main result in this paper  concerns finiteness of more general resolving subcategories of $\mod R$ and thick subcategories of $\lCM(R)$.
Denote the $n$-th syzygy of an $R$-module $M$ by $\syz^nM$.

\begin{thm}\label{1.2}
Let $R$ be a $d$-dimensional Cohen-Macaulay local ring with residue field $k$.
\begin{enumerate}[\rm(1)]
\item
Let $\X$ be a resolving subcategory of $\mod R$ containing $\syz^dk$ and strictly contained in $\CM(R)$.
If one of the following three statements holds, then $\X$ has infinite dimension.
\begin{itemize}
\item
$R$ is locally a hypersurface on the punctured spectrum.
\item
$R$ is locally with minimal multiplicity on the punctured spectrum.
\item
$R$ is excellent and locally of finite Cohen-Macaulay representation type on the punctured spectrum, and $\X$ contains a dualizing module.
\end{itemize}
\item
Let $R$ be Gorenstein and locally a hypersurface on the punctured spectrum.
Then every proper thick subcategory of $\lCM(R)$ containing $\syz^dk$ has infinite dimension.
\item
Let $R$ be a hypersurface.
Then every resolving subcategory of $\mod R$ containing a nonfree module and strictly contained in $\CM(R)$ and every nontrivial thick subcategory of $\lCM(R)$ have infinite dimension.
\end{enumerate}
\end{thm}

The third assertion of Theorem \ref{1.2} improves for hypersurfaces the main result of Oppermann and \v{S}\'{t}ov\'{i}\v{c}ek \cite{OS}.
Let $\Db(\mod R)$ denote the bounded derived category of $\mod R$ and $\perf R$ the full subcategory of perfect complexes.

\begin{cor}
Let $R$ be a local hypersurface.
Let $\X$ be a thick subcategory of $\Db(\mod R)$ with $\perf R\subsetneq\X\subsetneq\Db(\mod R)$.
Then the Verdier quotient $\X/\perf R$ has infinite dimension as a triangulated category.
In particular, the dimension of $\X$ is infinite.
\end{cor}

The organization of this paper is as follows.
In Section 2, together with our convention we will recall several basic definitions and fundamental facts for later use.
Section 3 will introduce the notions of the dimensions of subcategories of an abelian category and a triangulated category, and compare them with each other and with the concept of the radius of a subcategory which has been introduced in \cite{radius}. We also compute the dimension of the category of Cohen-Macaulay modules in some small cases, such as rational surface singularities, see Proposition \ref{36}. 
In Section 4, we will study the annihilators and supports of $\Tor$, $\Ext$ and $\lhom$ as functors on the direct product of given two subcategories $\X,\Y$ of $\mod R$ and $\lCM(R)$.
The results stated in this section will become the basis to obtain the main results of this paper.
Section 5 will mainly explore the nonfree loci of subcategories of $\CM(R)$ and the stable supports of subcategories of $\lCM(R)$, using the results obtained in Section 4.
In Section 6, we will consider finiteness of the dimensions of the resolving subcategory $\CM_0(R)$ of $\mod R$ and the thick subcategory $\lCM_0(R)$ of $\lCM(R)$, and give a proof of Theorem \ref{1.1}.
The aim of Section 7 will be to investigate finiteness of the dimensions of more general resolving subcategories of $\mod R$ and thick subcategories of $\lCM(R)$.
We will prove Theorem \ref{1.2} in this section.

\section{Preliminaries}

In this section, we recall basic definitions and fundamental facts for later use.

\begin{conv}
Throughout this paper, unless otherwise specified, we assume:\\
(1) All rings are commutative noetherian rings, and all modules are finitely generated.
All subcategories are nonempty, full and strict (i.e., closed under isomorphism).
Hence, the {\em subcategory} of a category $\C$ consisting of objects $\{M_\lambda\}_{\lambda\in\Lambda}$ always means the smallest strict full subcategory of $\C$ to which $M_\lambda$ belongs for all $\lambda\in\Lambda$.
This coincides with the full subcategory of $\C$ consisting of objects $X\in\C$ such that $X\cong M_\lambda$ for some $\lambda\in\Lambda$.\\
(2) Let $R$ be a ring.
The {\em singular locus} $\Sing R$ of $R$ is the set of prime ideals $\p$ of $R$ such that the local ring $R_\p$ is not regular.
By $\Spec_0(R)$ we denote the set of nonmaximal prime ideals of $R$.
This is nothing but the {\em punctured spectrum} of $R$ if $R$ is local. \new{For a prime $\p$, $V(\p)$ denotes the set of primes containing $\p$ in $\Spec(R)$.}\\
(3) The category of $R$-modules is denoted by $\mod R$, and the subcategory of modules of finite length is denoted by $\fl R$.
An $R$-module $M$ is called {\em (maximal) Cohen-Macaulay} if $\depth M_\p\ge\dim R_\p$ for all $\p\in\Spec R$.
(Hence the zero module is Cohen-Macaulay.)
The subcategory of $\mod R$ consisting of Cohen-Macaulay modules is denoted by $\CM(R)$.\\
(4) For a subcategory $\X$ of an additive category $\C$, we denote by $\add\X$ (or $\add_\C\X$, or $\add_R\X$ when $\C=\mod R$) the {\em additive closure} of $\X$, namely, the subcategory of $\C$ consisting of direct summands of finite direct sums of objects in $\X$.
When $\X$ consists of a single object $M$, we simply denote it by $\add M$ (or $\add_{\C}M$, $\add_RM$).
For an abelian category $\A$ with enough projective objects, we denote by $\proj\A$ the subcategory of projective objects.
For $n\ge1$ an $n$-th syzygy of an object $M\in\A$ is denoted by $\syz^nM$ (or $\syz_\A^nM$, or $\syz_R^nM$ when $\A=\mod R$). \new{This is not unique, but the choices differ only by projective summands.} 
Whenever $R$ is local and $\A=\mod R$, we use a {\em minimal free resolution} of $M$ to define $\syz^nM$, so that it is uniquely determined up to isomorphism.
\end{conv}

\begin{dfn}
Let $\A$ be an abelian category with enough projective objects.
A subcategory $\X$ of $\A$ is called {\it resolving} if $\X$ contains $\proj\A$ and is closed under direct summands, extensions and kernels of epimorphisms.
The last two closure properties mean that for an exact sequence $0 \to L \to M \to N \to 0$ in $\A$ with $N\in\X$ one has $L\in\X\Leftrightarrow M\in\X$.
\end{dfn}

The notion of a resolving subcategory has been introduced by Auslander and Bridger \cite{AB}.
It is a subcategory such that any two minimal resolutions of a module by modules in it have the same length (cf. \cite[Lemma (3.12)]{AB}).
Every resolving subcategory is closed under finite direct sums.
One can replace closure under kernels of epimorphisms with closure under syzygies (cf. \cite[Lemma 3.2]{Y}).
Clearly, $\proj\A$ and $\A$ are the smallest and largest resolving subcategories of $\A$, respectively.
A lot of resolving subcategories are known.
For example, $\CM(R)$ is a resolving subcategory of $\mod R$ if $R$ is Cohen-Macaulay.
The subcategory of $\mod R$ consisting of totally reflexive $R$-modules is resolving by \cite[(3.11)]{AB}.
One can construct a resolving subcategory easily by using vanishing of Tor or Ext.
Also, the modules of complexity less than a fixed integer form a resolving subcategory of $\mod R$.
For the details, we refer to \cite[Example 2.4]{res}.

\begin{dfn}
(1) For $R$-modules $M,N$ we set $\lhom_R(M,N)=\Hom_R(M,N)/\P_R(M,N)$, where $\P_R(M,N)$ is the set of $R$-homomorphisms $M\to N$ factoring through projective modules, which is an $R$-submodule of $\Hom_R(M,N)$.\\
(2) The {\em stable category} of $\CM(R)$, which is denoted by $\lCM(R)$, is defined by $\Ob(\lCM(R))=\Ob(\CM(R))$ and $\Hom_{\lCM(R)}(M,N)=\lhom_R(M,N)$ for $M,N\in\Ob(\lCM(R))$.
\end{dfn}

Let $R$ be a Gorenstein ring with $\dim R<\infty$.
Then $R$ is an {\em Iwanaga-Gorenstein} ring.
Taking the syzygy makes an autoequivalence $\syz:\lCM(R)\to\lCM(R)$ of categories, whose quasi-inverse is given by taking the cosyzygy, and $\lCM(R)$ is a triangulated category with shift functor $\sus=\syz^{-1}$.
For the details, see \cite[Theorem 4.4.1]{B} or \cite[\S2 in Chapter I]{H}.
We can also find in \cite[Remark 1.19]{stcm} how to define an exact triangle.

\begin{dfn}
A {\em thick} subcategory of a triangulated category is defined to be a triangulated subcategory closed under direct summands.
\end{dfn}

The notion of a thick subcategory has been introduced by Verdier \cite{V} to develop the theory of localizations of triangulated categories.
Every thick subcategory of a triangulated category $\T$ contains the zero object of $\T$, and is closed under shifts, namely, if $M$ is an object in $\X$, then so are $\sus M$ and $\sus^{-1}M$.
Clearly, $\{0\}$ and $\T$ are the smallest and largest thick subcategories of $\T$, respectively.
When $R$ is local, the bounded complexes of $R$-modules having finite complexity form a thick subcategory of the bounded derived category $\Db(\mod R)$ of $\mod R$.
When $R$ is Gorenstein with $\dim R<\infty$, for a fixed $R$-module $C$, the Cohen-Macaulay $R$-modules $M$ with $\Tor_i^R(M,C)=0$ for $i\gg0$ form a thick subcategory of $\lCM(R)$.

\begin{dfn}\label{defbar}
(1) For a subcategory $\X$ of $\CM(R)$, we define the category $\underline\X$ by $\Ob(\underline\X)=\Ob(\X)$ and $\Hom_{\underline\X}(M,N)=\lhom_R(M,N)$ for $M,N\in\Ob(\underline\X)$.\\
(2) For a subcategory $\X$ of $\lCM(R)$, we define the category $\overline\X$ by $\Ob(\overline\X)=\Ob(\X)$ and $\Hom_{\overline\X}(M,N)=\Hom_R(M,N)$ for $M,N\in\Ob(\overline\X)$.
\end{dfn}

Let $R$ be a Gorenstein ring of finite Krull dimension.
If $\X$ is a thick subcategory of $\lCM(R)$, then $\overline\X$ is a resolving subcategory of $\mod R$ contained in $\CM(R)$.
Conversely, if $\X$ is a resolving subcategory of $\mod R$ contained in $\CM(R)$, then $\underline\X$ is a thick subcategory of $\lCM(R)$ provided that $R$ is a local complete intersection; see \cite[Corollary 4.16]{radius}.

\begin{dfn}
(1) The {\em nonfree locus} $\NF(M)$ of an $R$-module $M$ is the set of prime ideals $\p$ of $R$ such that the $R_\p$-module $M_\p$ is nonfree.
The {\em nonfree locus} of a subcategory $\X$ of $\mod R$ is defined by $\NF(\X)=\bigcup_{M\in\X}\NF(M)$.
For a subset $W$ of $\Spec R$ we set $\NF_{\CM}^{-1}(W)=\{\,M\in\CM(R)\mid\NF(M)\subseteq W\,\}$.\\
(2) The {\em stable support} $\lSupp M$ of a Cohen-Macaulay $R$-module $M$ is the set of prime ideals $\p$ of $R$ such that $M_\p\cong 0$ in $\lCM(R_\p)$.
The {\em stable support} of a subcategory $\X$ of $\lCM(R)$ is defined by $\lSupp\X=\bigcup_{M\in\X}\lSupp M$.
For a subset $W$ of $\Spec R$ we set $\lSupp^{-1}W=\{\,M\in\lCM(R)\mid\lSupp M\subseteq W\,\}$.
\end{dfn}

Recall that a subset $W$ of $\Spec R$ is called {\em specialization-closed} if $W$ contains $\V(\p)$ for every $\p\in W$.
It is equivalent to saying that $W$ is a union of closed subsets of $\Spec R$.

\begin{rem}
The following hold for $M\in\mod R$, $N\in\CM(R)$, $\X\subseteq\mod R$, $\Y\subseteq\CM(R)$, $\Z\subseteq\lCM(R)$ and $W\subseteq\Spec R$ (cf. \cite[Propositions 1.14, 1.15, 6.2 and 6.4]{stcm}).\\
(1) $\NF(M)$ is empty if and only if $M$ is projective.
$\NF(M)$ contains only maximal ideals if and only if $M$ is locally free on $\Spec_0(R)$.\\
(2) $\NF(M)$ is a closed subset of $\Spec R$ in the Zariski topology.
$\NF(\X)$ is a specialization-closed subset of $\Spec R$.\\
(3) One has $\NF(\Y)\subseteq\Sing R$, $\NF(\NF_{\CM}^{-1}(W))\subseteq W$ and $\NF_{\CM}^{-1}(W)=\NF_{\CM}^{-1}(W\cap\Sing R)$.\\
(4) One has $\lSupp N=\NF(N)$, $\lSupp\underline\Y=\NF(\Y)$, $\lSupp\Z=\NF(\overline\Z)$ and $\lSupp^{-1}W=\underline{\NF_{\CM}^{-1}(W)}$.\\
(5) If $R$ is Cohen-Macaulay, then $\NF(\CM(R))=\Sing R$, and $\NF_{\CM}^{-1}(W)$ is a resolving subcategory of $\mod R$ contained in $\CM(R)$.\\
(6) If $R$ is Gorenstein with $\dim R<\infty$, then $\lSupp^{-1}W$ is a thick subcategory of $\lCM(R)$.
\end{rem}

\begin{dfn}
For an integer $n\ge-1$ we set $\CM_n(R)=\{\,M\in\CM(R)\mid\dim\NF(M)\le n\,\}$ and $\lCM_n(R)=\underline{\CM_n(R)}=\{\,M\in\lCM(R)\mid\dim\lSupp M\le n\,\}$.
\end{dfn}

\begin{rem}
Let $R$ be a $d$-dimensional Cohen-Macaulay local ring with residue field $k$.\\
(1) One has $\add R=\CM_{-1}(R)\subseteq\CM_0(R)\subseteq\CM_1(R)\subseteq\cdots\subseteq\CM_d(R)=\CM(R)$ and $\{0\} =\lCM_{-1}(R)\subseteq\lCM_0(R)\subseteq\lCM_1(R)\subseteq\cdots\subseteq\lCM_d(R)=\lCM(R)$.\\
(2) One has $\CM_n(R)=\NF_{\CM}^{-1}(\{\,\p\in\Sing R\mid\dim R/\p\le n\,\})$ and $\lCM_n(R)=\lSupp^{-1}(\{\,\p\in\Sing R\mid\dim R/\p\le n\,\})$ for $n\ge-1$.
Hence $\CM_n(R)$ is a resolving subcategory of $\mod R$ contained in $\CM(R)$, and $\lCM_n(R)$ is a thick subcategory of $\lCM(R)$ if $R$ is Gorenstein.\\
(3) The category $\CM_0(R)$ consists of the Cohen-Macaulay $R$-modules that are locally free on $\Spec_0(R)$.
Hence $\CM_0(R)$ is the smallest subcategory of $\mod R$ containing $\syz^dk$ that is closed under direct summands and extensions; see \cite[Corollary 2.6]{stcm}.
In particular, a resolving subcategory of $\mod R$ contains $\CM_0(R)$ if and only if it contains $\syz^dk$.
\end{rem}

\section{Definitions of dimensions of subcategories}

This section contains the key definitions and establishes several results.
More precisely, the notions of the dimensions of subcategories of an abelian category and a triangulated category will be introduced.
We will compare them with each other and with the concept of the radius of a subcategory.
Their relationships with representation types will also be explored.
First of all, we recall the definition of a ball given in \cite{BV,radius,R}.

\begin{dfn}
(1) Let $\T$ be a triangulated category.

(a) For a subcategory $\X$ of $\T$ we denote by $\langle\X\rangle$ the smallest subcategory of $\T$ containing $\X$ that is closed under finite direct sums, direct summands and shifts, i.e., $\langle\X\rangle=\add_\T\{\,\sus^iX\mid i\in\ZZ,\,X\in\X\,\}$.
When $\X$ consists of a single object $X$, we simply denote it by $\langle X\rangle$.

(b) For subcategories $\X,\Y$ of $\T$ we denote by $\X\ast\Y$ the subcategory of $\T$ consisting of objects $M$ which fit into an exact triangle $X \to M \to Y \to \sus X$ in $\T$ with $X\in\X$ and $Y\in\Y$.
We set $\X\diamond\Y=\langle\langle\X\rangle\ast\langle\Y\rangle\rangle$.

(c) Let $\C$ be a subcategory of $\T$.
We define the {\it ball of radius $r$ centered at $\C$} as
$$
\langle\C\rangle_r=
\begin{cases}
\langle\C\rangle & (r=1),\\
\langle\C\rangle_{r-1}\diamond\C=\langle\langle\C\rangle_{r-1}\ast\langle\C\rangle\rangle & (r\ge2).
\end{cases}
$$
If $\C$ consists of a single object $C$, then we simply denote it by $\langle C\rangle_r$, and call it the ball of radius $r$ centered at $C$.
We write $\langle\C\rangle_r^{\T}$ when we should specify that $\T$ is the ground category where the ball is defined.\\
(2) Let $\A$ be an abelian category with enough projective objects.

(a) For a subcategory $\X$ of $\A$ we denote by $[\X]$ the smallest subcategory of $\A$ containing $\proj\A$ and $\X$ that is closed under finite direct sums, direct summands and syzygies, i.e., $[\X]=\add_\A(\proj\A\cup\{\,\syz^iX\mid i\ge0,\,X\in\X\,\})$.
When $\X$ consists of a single object $X$, we simply denote it by $[X]$.

(b) For subcategories $\X,\Y$ of $\A$ we denote by $\X\circ\Y$ the subcategory of $\A$ consisting of objects $M$ which fit into an exact sequence $0 \to X \to M \to Y \to 0$ in $\A$ with $X\in\X$ and $Y\in\Y$.
We set $\X\bullet\Y=[[\X]\circ[\Y]]$.

(c) Let $\C$ be a subcategory of $\A$.
We define the {\it ball of radius $r$ centered at $\C$} as
$$
[\C]_r=
\begin{cases}
[\C] & (r=1),\\
[\C]_{r-1}\bullet\C=[[\C]_{r-1}\circ[\C]] & (r\ge2).
\end{cases}
$$
If $\C$ consists of a single object $C$, then we simply denote it by $[C]_r$, and call it the ball of radius $r$ centered at $C$.
We write $[\C]_r^{\A}$ when we should specify that $\A$ is the ground category where the ball is defined.
\end{dfn}

\new{We note that in the triangulated setting, the notion of a ball has been introduced (but not named) as $[C]_r$ in \cite{BV} and was called  the $r$th-thickening of $C$ in \cite{ABIM}.}

\begin{rem}
The following statements hold (cf. \cite{radius,R}).\\
(1) Let $\T$ be a triangulated category, and $\X,\Y,\Z,\C$ subcategories.

(a) An object $M\in\T$ belongs to $\X\diamond\Y$ if and only if there is an exact triangle $X \to Z \to Y \to \sus X$ with $X\in\langle\X\rangle$ and $Y\in\langle\Y\rangle$ such that $M$ is a direct summand of $Z$.

(b) One has $(\X\diamond\Y)\diamond\Z=\X\diamond(\Y\diamond\Z)$ and $\langle\C\rangle_a\diamond\langle\C\rangle_b=\langle\C\rangle_{a+b}$ for all $a,b>0$.\\
(2) Let $\A$ be an abelian category with enough projectives, and $\X,\Y,\Z,\C$ subcategories.

(a) An object $M\in\A$ belongs to $\X\bullet\Y$ if and only if there is an exact sequence $0 \to X \to Z \to Y \to 0$ with $X\in[\X]$ and $Y\in[\Y]$ such that $M$ is a direct summand of $Z$.

(b) One has $(\X\bullet\Y)\bullet\Z=\X\bullet(\Y\bullet\Z)$ and $[\C]_a\bullet[\C]_b=[\C]_{a+b}$ for all $a,b>0$.
\end{rem}

Now, for a triangulated category and an abelian category with enough projective objects, we can make the definitions of the dimensions of subcategories.

\begin{dfn}
(1) Let $\T$ be a triangulated category.
Let $\X$ be a subcategory of $\T$.
We define the {\it dimension} of $\X$, denoted by $\dim\X$ (or $\dim_\T\X$), as the infimum of the integers $n\ge0$ such that $\X=\langle G\rangle_{n+1}^\T$ for some $G\in\X$.\\
(2) Let $\A$ be an abelian category with enough projective objects.
Let $\X$ be a subcategory of $\A$.
We define the {\it dimension} of $\X$, denoted by $\dim\X$ (or $\dim_\A\X$), as the infimum of the integers $n\ge0$ such that $\X=[G]_{n+1}^\A$ for some $G\in\X$.
\end{dfn}

\begin{rem}
(1) One has $\dim\X\in\N\cup\{\infty\}$ in both senses.\\
(2)If $\X$ is a triangulated subcategory of $\T$ (respectively, an abelian subcategory of $\A$ containing $\proj\A$), then $\dim_\T\X=\dim_\X\X$ (respectively, $\dim_\A\X=\dim_\X\X$).\\
(3) The definition itself works for every subcategory $\X$ of $\T$ (respectively, $\A$).
But the equality $\X=\langle G\rangle_{n+1}^\T$ (respectively, $\X=[G]_{n+1}^\A$) forces $\X$ to be closed under finite direct sums, direct summands and shifts (respectively, to contain the projective objects and be closed under finite direct sums, direct summands and syzygies).
So, basically, a subcategory whose dimension is considered is supposed to be thick (respectively, resolving).\\
(4) The subcategory $\{0\}$ of $\T$ and the subcategory $\proj\A$ of $\A$ have dimension $0$.\\

\end{rem}


\new {A very similar notion of a radius of a subcategory was given in  \cite[Definition 2.3]{radius}. Indeed, one defines the {\it radius} of $\X$, denoted by $\radi\X$ (or $\radi_\A\X$), as the infimum of the integers $n\ge0$ such that $\X \subseteq  [G]_{n+1}^\A$ (may not be equal) for some $G\in\X$.  (In fact it can be defined for a subcategory of an arbitrary abelian category with enough projective objects). The difference could be subtle, see  below.}

\begin{prop}\label{35}
\begin{enumerate}[\rm(1)]
\item
One has $\radi\X\le\dim\X$ for any subcategory $\X$ of $\mod R$.
\item
Let $R$ be Gorenstein of finite dimension.
Let $\X$ be a thick subcategory of $\lCM(R)$.
Then $\dim\X\le\dim\overline\X$ (see Definition \ref{defbar}) holds. We also have $\radi\CM(R)=\dim\CM(R)$.
\item
If $R$ is a local hypersurface, then one has $\dim\lCM(R)=\radi\CM(R)=\dim\CM(R)$.
\end{enumerate}
\end{prop}

\begin{proof}
(1) This assertion is by definition.

(2) We claim that for a Cohen-Macaulay $R$-module $G$ and an integer $n\ge1$ every object $M\in[G]_n^{\mod R}$ belongs to ${\langle G\rangle}_n^{\lCM(R)}$.
Indeed, this claim is an easy consequence by induction on $n$.
Now assume $\dim\overline\X=m<\infty$.
Then there is a Cohen-Macaulay $R$-module $G$ satisfying $\overline\X=[G]_{m+1}^{\mod R}$.
The claim implies that $\X$ is contained in ${\langle G\rangle}_{m+1}^{\lCM(R)}$.
Since $\X$ is thick, it coincides with ${\langle G\rangle}_{m+1}^{\lCM(R)}$.
Hence we have $\dim\X\le m$.

For the second assertion, suppose $\radi\CM(R)=n$, so there exists $G \in \mod R$ such that $\CM(R) \subseteq [G]_{n+1}$. It follows that for $a$ big enough one has $\CM(R) = [\syz^{-a}\syz^aG]_{n+1}$, thus $\dim \CM(R) \leq n$ and the equality follows from (1). 

(3) The inequalities $\dim\lCM(R)\le\radi\CM(R)\le\dim\CM(R)$ are obtained by using \cite[Proposition 2.6(1)]{radius} and (1).
The proof of \cite[Proposition 2.6(2)]{radius} actually shows that $\dim\CM(R)\le\dim\lCM(R)$.
\end{proof}

\begin{rem}
The dimension of a subcategory $\X$ of $\A$ is by definition the infimum of $n\ge0$ with $\X\subseteq[G]_{n+1}$ for some $G\in\X$.
Then the only difference between the definitions of $\dim\X$ and $\radi\X$ is that we do now require the object $G$ to be in $\X$.
This is subtle but will turn out to be crucial.
For example, let $R=\CC[[x,y]]/(x^2y)$.
Then the radius of $\CM_0(R)$ is $1$ by \cite[Proposition 4.2]{BGS} and \cite[Propositions 2.10]{radius}, in particular, it is finite.
But the dimension of $\CM_0(R)$ will turn out to be infinite by Theorem \ref{main}.
\end{rem}

Next we calculate some  examples of categories with small dimensions. 
Recall that a Cohen-Macaulay local ring $R$ is said to have {\it finite} (respectively, {\it countable}) {\it Cohen-Macaulay representation type} if there are only finitely (respectively, countably but infinitely) many nonisomorphic indecomposable Cohen-Macaulay $R$-modules.

\begin{prop}\label{36}Let $(R,\m,k)$ be a Cohen-Macaulay  local ring.

\begin{enumerate}[\rm(1)]
\item
If  $R$ has finite Cohen-Macaulay representation type then $\dim\CM(R)=0$. The converse is true of $R$ is hensenlian and Gorenstein.
\item Suppose $\dim R=2$, $k$ is algebraically closed  and $R$ is hensenlian, normal with rational singularity in the sense of \cite{Lip}. Then $\dim\CM(R)\leq 1$.
\item
Suppose  $R$ is a complete local hypersurface with an algebraically closed coefficient field of characteristic not two.
If $R$ has countable Cohen-Macaulay representation type, then one has $\dim\CM(R)=1$.
\end{enumerate}
\end{prop}

\begin{proof}

(1) This follows from (1) of Proposition \ref{35} and \cite[Proposition 2.8]{radius}.

(2) By \cite[Theorem 3.6]{IW1} (which rests on \cite[Theorem 2.1]{Wu}, whose proof goes through in our slightly more general setting), there exists $X \in \CM(R)$ such that $\mathsf{\Omega CM}(R) = \add X$.
Let $M\in \CM(R)$, and let $M^{\vee}$ denote $\Hom_R(M,\omega_R)$. We have an exact sequence $0 \to \syz M^{\vee} \to R^n \to  M^{\vee} \to 0 $.
Applying $(-)^\vee$ we get an exact sequence $0 \to M  \to \omega_R^n \to  (\syz M^{\vee})^{\vee} \to 0 $.
It follows that $\CM(R) = [X^{\vee} \oplus \omega_R]_2$. 

(3) This is shown by (3) of Proposition \ref{35} and \cite[Proposition 2.10]{radius}.
\end{proof}

\section{Annihilators and supports of $\Tor$, $\Ext$ and $\lhom$}

In this section, we investigate the annihilators and supports of $\Tor$, $\Ext$ and $\lhom$ as functors on the direct product of given two subcategories $\X,\Y$ of $\mod R$ and $\lCM(R)$.
Our results stated in this section will be the basis to obtain the main results of this paper.
We start by fixing our notation.

\begin{nota}
(1) For subcategories $\X,\Y$ of $\mod R$, we define $\Tor(\X,\Y)=\bigoplus_{i>0,\,X\in\X,\,Y\in\Y}\Tor_i^R(X,Y)$ and $\Ext(\X,\Y)=\bigoplus_{i>0,\,X\in\X,\,Y\in\Y}\Ext^i_R(X,Y)$.
If $\X$ (respectively, $\Y$) consists of a single module $M$, we simply write $\Tor(M,\Y)$ and $\Ext(M,\Y)$ (respectively, $\Tor(\X,M)$ and $\Ext(\X,M)$).\\
(2) For subcategories $\X,\Y$ of $\lCM(R)$, we define $\lhom(\X,\Y) = \bigoplus_{X\in\X,\,Y\in\Y}\lhom_R(X,Y)$.
If $\X$ (respectively, $\Y$) consists of a single module $M$, we simply write $\lhom(M,\Y)$ (respectively, $\lhom(\X,M)$).
\end{nota}

\begin{rem}
(1) Let $\X\subseteq\X'$ and $\Y\subseteq\Y'$ be subcategories of $\mod R$.
Then
\begin{align*}
& \Supp\Tor(\X,\Y) \subseteq \Supp\Tor(\X',\Y'),\quad
\Supp\Ext(\X,\Y)\subseteq \Supp\Ext(\X',\Y'),\\
& \V(\Ann\Tor(\X,\Y)) \subseteq \V(\Ann\Tor(\X',\Y')),\quad
\V(\Ann\Ext(\X,\Y))\subseteq \V(\Ann\Ext(\X',\Y')).
\end{align*}
(2) Let $\X,\Y$ be subcategories of $\mod R$.
Then one has
$$
\Supp\Tor(\X,\Y)\subseteq\V(\Ann\Tor(\X,\Y)),\quad\Supp\Ext(\X,\Y)\subseteq\V(\Ann\Ext(\X,\Y)).
$$
The equalities do not hold in general because $\Tor(\X,\Y)$ and $\Ext(\X,\Y)$ are usually infinitely generated $R$-modules.\\
(3) Let $R$ be Gorenstein with $\dim R<\infty$, and let $\X,\Y$ be subcategories of $\lCM(R)$.
Suppose that either $\X$ or $\Y$ is closed under shifts.
Then one has the equalities
$$
\Ann\lhom(\X,\Y)=\Ann\Ext(\overline\X,\overline\Y),\quad\Supp\lhom(\X,\Y)=\Supp\Ext(\overline\X,\overline\Y).
$$
Indeed, for all $i>0$, $X\in\X$ and $Y\in\Y$ we have $\Ext_R^i(X,Y)\cong\lhom_R(\sus^{-i}X,Y)\cong\lhom_R(X,\sus^iY)$ and $\lhom_R(X,Y)\cong\Ext_R^1(\sus X,Y)\cong\Ext_R^1(X,\sus^{-1}Y)$.
The assertion is an easy consequence of these isomorphisms.
Using these two equalities, we can translate results on $\Ann\Ext$ and $\Supp\Ext$ into ones on $\Ann\lhom$ and $\Supp\lhom$.
\end{rem}

Our first purpose in this section is to analyze the annihilators of $\Tor,\Ext$ on subcategories of $\mod R$ by means of the annihilators of $\Tor,\Ext$ on smaller subcategories:

\begin{prop}\label{a^2}
Let $R$ be local and $M$ be an $R$-module.
Let $a\in R$, $n\in\ZZ$ and $t\in\N$.
\begin{enumerate}[\rm(1)]
\item
Suppose that $a\Tor_n^R(M,X)=a\Tor_{n-1}^R(M,X)=0$ for all $R$-modules $X$ with $\dim X\le t$.
Then $a^2\Tor_n^R(M,X)=0$ for all $R$-modules $X$ with $\dim X\le t+1$.
\item
Suppose that $a\Ext_R^n(M,X)=a\Ext_R^{n+1}(M,X)=0$ for all $R$-modules $X$ with $\dim X\le t$.
Then $a^2\Ext_R^n(M,X)=0$ for all $R$-modules $X$ with $\dim X\le t+1$.
\item
Suppose that $a\Ext_R^n(X,M)=a\Ext_R^{n+1}(X,M)=0$ for all $R$-modules $X$ with $\dim X\le t$.
Then $a^2\Ext_R^n(X,M)=0$ for all $R$-modules $X$ with $\dim X\le t+1$.
\end{enumerate}
\end{prop}

\begin{proof}
We only prove the first assertion, since the second and third assertions are shown similarly.
Fix an $R$-module $X$ with $\dim X\le t+1$.
We want to show $a^2\Tor_n^R(M,X)=0$.
By assumption, we have only to deal with the case $\dim X=t+1$.
Let $r\in R$ be part of a system of parameters of $X$.
Then we have $\dim X/rX=t$, and it is easy to see that $\dim(0:_Xr)\le t$ holds.
Our assumption implies $a\Tor_i^R(M,X/rX)=a\Tor_i^R(M,(0:_Xr))=0$ for $i=n,n-1$.
There are exact sequences $0 \to (0:_Xr) \to X \to rX \to 0$ and $0 \to rX \to X \to X/rX \to 0$, which give exact sequences
\begin{align}
&\label{colon} \Tor_n^R(M,X) \xrightarrow{f} \Tor_n^R(M,rX) \to \Tor_{n-1}^R(M,(0:_Xr)),\\
&\label{modulo} \Tor_n^R(M,rX) \xrightarrow{g} \Tor_n^R(M,X) \to \Tor_n^R(M,X/rX).
\end{align}
Let $y\in\Tor_n^R(M,X)$.
By \eqref{modulo} we have $ay=g(z)$ for some $z\in\Tor_n^R(M,rX)$, and by \eqref{colon} we have $az=f(w)$ for some $w\in\Tor_n^R(M,X)$.
Hence $a^2y=gf(w)=rw$, and we obtain $a^2\Tor_n^R(M,X)\subseteq r\Tor_n^R(M,X)$ for every element $r\in R$ that is part of system of parameters of $M$.
Since the element $r^j$ is also part of system of parameters of $M$ for all $j>0$, the module $a^2\Tor_n^R(M,X)$ is contained in $\bigcap_{j>0}r^j\Tor_n^R(M,X)$, which is zero by Krull's intersection theorem.
\end{proof}

Iteration of the above proposition yields the following result; the annihilators of $\Tor,\Ext$ on $\mod R$ can be controlled by the annihilators of $\Tor,\Ext$ on $\fl R$.

\begin{cor}\label{22d}
Let $R$ be a local ring of dimension $d$.
Let $a\in R$ and $n\in\ZZ$.
\begin{enumerate}[\rm(1)]
\item
Suppose that $a\Tor_i^R(M,N)=0$ for all $n-2d\le i\le n$ and $M,N\in\fl R$.
Then $a^{2^{2d}}\Tor_n^R(M,N)=0$ for all $M,N\in\mod R$.
\item
Suppose that $a\Ext_R^i(M,N)=0$ for all $n\le i\le n+2d$ and $M,N\in\fl R$.
Then $a^{2^{2d}}\Ext_R^n(M,N)=0$ for all $M,N\in\mod R$.
\end{enumerate}
\end{cor}

\begin{proof}
Let us show the first assertion; the second one follows from a similar argument.

First, fix $M\in\fl R$.
Applying Proposition \ref{a^2}(1) repeatedly, we get:
\begin{align*}
& \text{$a\Tor_i^R(M,N)=0$ for all $n-2d\le i\le n$ and $N\in\fl R$},\\
& \text{$a^2\Tor_i^R(M,N)=0$ for all $n-2d+1\le i\le n$ and $N\in\mod R$ with $\dim N\le1$},\\
& \text{$a^4\Tor_i^R(M,N)=0$ for all $n-2d+2\le i\le n$ and $N\in\mod R$ with $\dim N\le2$},\\
& \qquad\cdots
\end{align*}
and we obtain $a^{2^d}\Tor_i^R(M,N)=0$ for all $n-d\le i\le n$ and $M\in\fl R$ and $N\in\mod R$.

Next, fix $N\in\mod R$.
A similar argument to the above gives:
\begin{align*}
& \text{$a^{2^d}\Tor_i^R(M,N)=0$ for all $n-d\le i\le n$ and $M\in\fl R$},\\
& \text{$a^{2^{d+1}}\Tor_i^R(M,N)=0$ for all $n-d+1\le i\le n$ and $M\in\mod R$ with $\dim M\le1$},\\
& \text{$a^{2^{d+2}}\Tor_i^R(M,N)=0$ for all $n-d+2\le i\le n$ and $M\in\mod R$ with $\dim M\le2$},\\
& \qquad\cdots
\end{align*}
and finally we get $a^{2^{2d}}\Tor_n^R(M,N)=0$ for all $M,N\in\mod R$.
\end{proof}

In the case where $R$ is Cohen-Macaulay, the annihilators of $\Tor,\Ext$ on $\mod R$ can also be controlled by the annihilators of $\Tor,\Ext$ on $\CM_0(R)$.

\begin{prop}\label{4d}
Let $R$ be a $d$-dimensional Cohen-Macaulay local ring.
\begin{enumerate}[\rm(1)]
\item
Let $a\in\Ann\Tor(\CM_0(R),\CM_0(R))$.
Then $a^{2^{2d}}\Tor_i^R(M,N)=0$ for all $i>4d$ and all $R$-modules $M,N$.
\item
Let $a\in\Ann\Ext(\CM_0(R),\CM_0(R))$.
Then $a^{2^{2d}(d+1)}\Ext_R^i(M,N)=0$ for all $i>d$ and all $R$-modules $M,N$.
\end{enumerate}
\end{prop}

\begin{proof}
Let $M,N$ be $R$-modules of finite length.
Note that $\syz^dM,\syz^dN$ belong to $\CM_0(R)$.

(1) We have $a\Tor_i^R(\syz^dM,\syz^dN)=0$ for every $i>0$, which implies $a\Tor_i^R(M,N)=0$ for every $i>2d$.
Now let $n>4d$ be an integer.
Then $a\Tor_i^R(M,N)=0$ for all integers $i$ with $n-2d\le i\le n$.
It follows from Corollary \ref{22d}(1) that $a^{2^{2d}}\Tor_n^R(X,Y)=0$ for all $R$-modules $X,Y$.

(2) Fix an integer $i>0$.
We have $a\Ext_R^i(K,L)=0$ for all $K,L\in\CM_0(R)$.
For each integer $j\ge0$ there is an exact sequence $0 \to \syz^{j+1}N \to F_j \to \syz^jN \to 0$ such that $F_j$ is free.
Since $F_j$ and $\syz^dN$ belong to $\CM_0(R)$, we have $a\Ext_R^i(K,F_j)=a\Ext_R^i(K,\syz^dN)=0$.
An exact sequence $\Ext_R^i(K,F_j) \to \Ext_R^i(K,\syz^jN) \to \Ext_R^{i+1}(K,\syz^{j+1}N)$ is induced, and an inductive argument shows that $a^{d+1}\Ext_R^i(K,N)=0$.
(Note that in general an exact sequence $A\xrightarrow{\alpha}B\xrightarrow{\beta}C$ yields $\Ann A\cdot\Ann C\subseteq\Ann B$.)
Letting $K:=\syz^dM$, we observe that $a^{d+1}\Ext_R^h(M,N)=0$ for every $h>d$.
Corollary \ref{22d}(2) yields $(a^{d+1})^{2^{2d}}\Ext_R^h(X,Y)=0$ for all $h>d$ and $X,Y\in\mod R$.
\end{proof}

Now we can prove that the annihilators of $\Tor,\Ext$ on $\CM_0(R)$ are contained in all prime ideals in the singular locus of $R$.

\begin{prop}\label{5.1}
\begin{enumerate}[\rm(1)]
\item
Let $R$ be a Cohen-Macaulay local ring.
Then one has
$$
\Sing R\subseteq
\V(\Ann\Tor(\CM_0(R),\CM_0(R)))\cap
\V(\Ann\Ext(\CM_0(R),\CM_0(R))).
$$
\item
Let $R$ be a Gorenstein local ring.
Then
$$
\Sing R\subseteq\V(\Ann\lhom(\lCM_0(R),\lCM_0(R))).
$$
\end{enumerate}
\end{prop}

\begin{proof}
(1) Let $\p$ be any prime ideal in $\Sing R$.
Take an element $a\in\Ann\Tor(\CM_0(R),\CM_0(R))$.
Then, by Proposition \ref{4d}(1) we have $a^{2^{2d}}\Tor_i^R(R/\p,R/\p)=0$ for $i>4d$.
Localization at $\p$ shows that $a^{2^{2d}}\Tor_i^{R_\p}(\kappa(\p),\kappa(\p))=0$ for $i>4d$.
If $a$ is not in $\p$, then $a^{2^{2d}}$ is a unit in $R_\p$, and it follows that $\Tor_i^{R_\p}(\kappa(\p),\kappa(\p))=0$ for $i>4d$.
This is impossible since the local ring $R_\p$ is nonregular, and thus $a\in\p$.
The assertion for $\Ext$ is also proved analogously.

(2) Since $\Ann\lhom(\lCM_0(R),\lCM_0(R))$ coincides with $\Ann\Ext(\CM_0(R),\CM_0(R))$, the assertion follows from (1).
\end{proof}

\begin{rem}
In some results such as Propositions \ref{5.1} and \ref{4.4} the stable category versions are given, because they are used in the proofs of the main results of this paper.
We can also give the stable category versions of other results such as Propositions \ref{4d} and \ref{cep}, but do not, just because they are not necessary to prove our main results.
\end{rem}

Let $R$ be a complete equicharacteristic local ring with residue field $k$.
Let $A$ be a {\em Noether normalization} of $R$, that is, a formal power series subring $k[[x_1,\dots,x_d]]$, where $x_1,\dots,x_d$ is a system of parameters of $R$.
Let $R^e=R\otimes_AR$ be the enveloping algebra of $R$ over $A$.
Define a map $\mu:R^e\to R$ by $\mu(x\otimes y)=xy$ for $x,y\in R$, and put $\NN^R_A=\mu(\Ann_{R^e}\Ker\mu)$.
Then $\NN^R_A$ is an ideal of $R$, which is called the {\em Noether different} of $R$ over $A$.
We denote by $\NN^R$ the sum of $\NN^R_A$, where $A$ runs through the Noether normalizations of $R$.
Under a mild assumption, we can substantially refine the statement for $\Ext$ in the previous proposition:

\begin{prop}\label{cep}
Let $R$ be a Cohen-Macaulay complete equicharacteristic local ring with perfect residue field.
Then $\V(\Ann\Ext(\X,\Y))=\Sing R$ for all $\CM_0(R)\subseteq\X\subseteq\CM(R)$ and $\CM_0(R)\subseteq\Y\subseteq\mod R$.
\end{prop}

\begin{proof}
We have the inclusions below, the first of which follows from Proposition \ref{5.1}(1).
\begin{align*}
\Sing R
& \subseteq\V(\Ann\Ext(\CM_0(R),\CM_0(R)))\\
& \subseteq\V(\Ann\Ext(\X,\Y))\subseteq\V(\Ann\Ext(\CM(R),\mod R)).
\end{align*}
By virtue of \cite[Corollary 5.13]{W}, the module $\Ext(\CM(R),\mod R)$ is annihilated by the ideal $\NN^R$.
Hence $\V(\Ann\Ext(\CM(R),\mod R))$ is contained in $\V(\NN^R)$.
On the other hand, it follows from \cite[Lemma (6.12)]{Y} that $\V(\NN^R)$ coincides with $\Sing R$.
\end{proof}

In the rest of this section, we study over an arbitrary local ring how the sets $\V(\Ann\Tor(\X,\Y))$ and $\V(\Ann\Ext(\X,\Y))$ vary as $\X,\Y$ move subcategories of $\mod R$.

\begin{prop}\label{doko}
Let $R$ be a local ring of dimension $d$, and let $\X$ be a subcategory of $\mod R$.
Then one has
\begin{align*}
& \V(\Ann\Tor(\X,\syz^d\mod R))\subseteq\V(\Ann\Tor(\X,\fl R))\subseteq\V(\Ann\Tor(\X,\mod R)),\\
& \V(\Ann\Ext(\X,\fl R))=\V(\Ann\Ext(\X,\mod R)),\\
& \V(\Ann\Ext(\fl R,\X))=\V(\Ann\Ext(\mod R,\X)),
\end{align*}
where $\syz^d\mod R$ denotes the subcategory of $\mod R$ consisting of the $d$-th syzygies of modules in $\X$.
\end{prop}

\begin{proof}
Let $a,b,c\in R$ be elements such that $a\Tor_i^R(X,M)=b\Ext_R^i(X,M)=c\Ext_R^i(M,X)=0$ for all $i>0$, $X\in\X$ and $M\in\fl R$.
Then a similar argument to the proof of Corollary \ref{22d} shows that $a^{2^d}\Tor_i^R(X,M)=b^{2^d}\Ext_R^j(X,M)=c^{2^d}\Ext_R^j(M,X)=0$ for all $i>d$, $j>0$, $X\in\X$ and $M\in\mod R$.
Hence we have
\begin{align*}
\Ann\Tor(\X,\fl R) & \subseteq\sqrt{\Ann\Tor(\X,\syz^d\mod R)},\\
\Ann\Ext(\X,\fl R) & \subseteq\sqrt{\Ann\Ext(\X,\mod R)},\\
\Ann\Ext(\fl R,\X) & \subseteq\sqrt{\Ann\Ext(\mod R,\X)}.
\end{align*}
Therefore $\V(\Ann\Tor(\X,\syz^d\mod R))$, $\V(\Ann\Ext(\X,\mod R))$ and $\V(\Ann\Ext(\mod R,\X))$ are contained in $\V(\Ann\Tor(\X,\fl R))$, $\V(\Ann\Ext(\X,\fl R))$ and $\V(\Ann\Ext(\fl R,\X))$, respectively.
The other inclusion relations are straightforward.
\end{proof}

\begin{cor}
Let $R$ be a local ring.
The sets $\V(\Ann\Ext(\X,\Y))$ are constant over the subcategories $\X,\Y$ of $\mod R$ containing $\fl R$.
\end{cor}

\begin{proof}
First, since $\fl R\subseteq\Y\subseteq\mod R$, we have $\V(\Ann\Ext(\X,\fl R))\subseteq\V(\Ann\Ext(\X,\Y))\subseteq\V(\Ann\Ext(\X,\mod R))$.
Proposition \ref{doko} implies that the left and right ends are equal.
Next, the inclusions $\fl R\subseteq\X\subseteq\mod R$ imply $\V(\Ann\Ext(\fl R,\mod R))\subseteq\V(\Ann\Ext(\X,\mod R))\subseteq\V(\Ann\Ext(\mod R,\mod R))$, where the left and right ends coincide by Proposition \ref{doko}.
Thus, we obtain $\V(\Ann\Ext(\X,\Y))=\V(\Ann\Ext(\X,\mod R))=\V(\Ann\Ext(\mod R,\mod R))$.
\end{proof}

\section{Nonfree loci and stable supports of subcategories}

This section mainly investigates the nonfree loci of subcategories of $\CM(R)$ and the stable supports of subcategories of $\lCM(R)$, using the results obtained in the previous section.
First of all, let us study relationships of nonfree loci and stable supports with supports and annhilators of $\Tor$, $\Ext$ and $\lhom$.

\begin{prop}\label{4.4}
\begin{enumerate}[\rm(1)]
\item
For an $R$-module $M$ one has equalities
\begin{align*}
\NF(M) 
& = \Supp\Tor(M,\mod R) = \V(\Ann\Tor(M,\mod R))\\
& = \Supp\Ext(M,\mod R) = \V(\Ann\Ext(M,\mod R)).
\end{align*}
\item
Let $R$ be a $d$-dimensional Gorenstein ring.
For an object $M\in\lCM(R)$ one has
$$
\lSupp M = \Supp\lhom(M,\lCM(R)) = \V(\Ann\lhom(M,\lCM(R))).
$$
\end{enumerate}
\end{prop}

\begin{proof}
(1) We prove $\NF(M)=\Supp\Tor(M,\mod R)=\V(\Ann\Tor(M,\mod R))$.
First, let $\p\in\NF(M)$.
Then we have $\Tor_1^R(M,R/\p)_\p=\Tor_1^{R_\p}(M_\p,\kappa(\p))\ne0$.
Hence $\NF(M)$ is contained in $\Supp\Tor(M,\mod R)$, which is contained in $\V(\Ann\Tor(M,\mod R))$.
Next, let $\p\in\V(\Ann\Tor(M,\mod R))$, and suppose $\p\notin\NF(M)$.
Then $M_\p\cong R_\p^{\oplus n}$ for some $n\ge0$, which gives an exact sequence $0 \to K \to M \xrightarrow{f} R^n \to C \to 0$ such that $K_\p=0=C_\p$.
Choose $x\in R\setminus\p$ with $xK=0=xC$.
Taking the image $L$ of $f$, we have an exact sequence
\begin{equation}\label{tttt}
\Tor_i^R(K,N) \to \Tor_i^R(M,N) \to \Tor_i^R(L,N)=\Tor_{i+1}^R(C,N)
\end{equation}
for each $i>0$ and $N\in\mod R$.
Since $x$ annihilates $\Tor_i^R(K,N)$ and $\Tor_{i+1}^R(C,N)$, the element $x^2$ annihilates $\Tor_i^R(M,N)$.
This means that $x^2$ belongs to $\Ann\Tor(M,\mod R)$, which is contained in $\p$.
Thus $x$ belongs to $\p$, which is a contradiction.
Consequently, $\V(\Ann\Tor(M,\mod R))$ is contained in $\NF(M)$, and we conclude $\NF(M)=\Supp\Tor(M,\mod R)=\V(\Ann\Tor(M,\mod R))$.
Along the same lines as in the above, one can prove the equalities $\NF(M)=\Supp\Ext(M,\mod R)=\V(\Ann\Ext(M,\mod R))$, using the following instead of \eqref{tttt}:
$$
\Ext^{i+1}_R(C,N)=\Ext^i_R(L,N) \to \Ext^i_R(M,N) \to \Ext^i_R(K,N).
$$

(2) We have $\lSupp M=\NF(M)=\Supp\Ext(M,\mod R)=\V(\Ann\Ext(M,\mod R))$ by (1).
Since $R$ is a Gorenstein ring and $M$ is a Cohen-Macaulay module, the isomorphism $\Ext_R^i(M,N)\cong\Ext_R^{i+d}(M,\syz^dN)$ holds for all $i>0$ and $N\in\mod R$.
Note here that $\syz^dN$ is Cohen-Macaulay.
Now it is easy to observe that the equalities $\Supp\Ext(M,\mod R)=\Supp\Ext(M,\CM(R))=\Supp\lhom(M,\lCM(R))$ and $\V(\Ann\Ext(M,\mod R))=\V(\Ann\Ext(M,\CM(R)))=\V(\Ann\lhom(M,\lCM(R)))$ hold.
\end{proof}

\begin{rem}
\begin{enumerate}[(1)]
\item
Let $M$ be an $R$-module.
Take an ideal $I$ of $R$ with $\NF(M)=\V(I)$.
Then Proposition \ref{4.4}(1) implies that there exists an integer $h>0$ such that $I^h$ annihilates $\Tor_i^R(M,X)$ and $\Ext_R^i(M,X)$ for all $i>0$ and $X\in\mod R$.
This is a generalization of \cite[Lemma 4.3]{DV}.
\item
One has $\NF(\X)=\Supp\Ext(\X,\Y)$ for all subcategories $\X,\Y$ of $\mod R$ with $\syz\X\subseteq\Y$.
In fact, it is obvious that $\NF(\X)$ contains $\Supp\Ext(\X,\Y)$, and the opposite inclusion relation is obtained by the fact that $\NF(X)=\Supp\Ext_R^1(X,\syz X)$ for each $R$-module $X$ (cf. \cite[Proposition 2.10]{res}).
The equality $\NF(M) = \Supp\Ext(M,\mod R)$ in Proposition \ref{4.4} is also a consequence of this statement.
\end{enumerate}
\end{rem}

Here we need to inspect the annihilators of $\Tor,\Ext,\lhom$ of balls:

\begin{lem}\label{hos}
\begin{enumerate}[\rm(1)]
\item
Let $\X,\Y$ be subcategories of $\mod R$ and $n\ge0$ an integer.
Then:
\begin{align*}
& (\Ann\Tor(\X,\Y))^n\subseteq
\left\{
\begin{gathered}
\Ann\Tor([\X]_n,\Y)\\
\Ann\Tor(\X,[\Y]_n)
\end{gathered}
\right\}
\subseteq\Ann\Tor(\X,\Y),\\
& (\Ann\Ext(\X,\Y))^n\subseteq
\left\{
\begin{gathered}
\Ann\Ext([\X]_n,\Y)\\
\Ann\Ext(\X,[\Y]_n)
\end{gathered}
\right\}
\subseteq\Ann\Ext(\X,\Y).
\end{align*}
In particular, $\V(\Ann\Tor(\X,\Y))=\V(\Ann\Tor([\X]_n,\Y))=\V(\Ann\Tor(\X,[\Y]_n))$ and $\V(\Ann\Ext(\X,\Y))=\V(\Ann\Ext([\X]_n,\Y))=\V(\Ann\Ext(\X,[\Y]_n))$ hold.
\item
Suppose that $R$ is Gorenstein of finite Krull dimension.
Let $\X,\Y$ be subcategories of $\lCM(R)$ and $n\ge0$ an integer.
Then there are inclusions
\begin{align*}
& (\Ann\lhom(\X,\Y))^n\subseteq
\left\{
\begin{gathered}
\Ann\lhom(\langle\X\rangle_n,\Y)\\
\Ann\lhom(\X,\langle\Y\rangle_n)
\end{gathered}
\right\}
\subseteq\Ann\lhom(\X,\Y).
\end{align*}
Hence $\V(\Ann\lhom(\X,\Y))=\V(\Ann\lhom(\langle\X\rangle_n,\Y))=\V(\Ann\lhom(\X,\langle\Y\rangle_n))$.
\end{enumerate}
\end{lem}

\begin{proof}
Let us only prove the inclusions $(\Ann\Tor(\X,\Y))^n\subseteq\Ann\Tor([\X]_n,\Y)\subseteq\Ann\Tor(\X,\Y)$; the other inclusions can be shown similarly.
It is clear that the second inclusion holds.
As for the first inclusion, it suffices to show that
$$
\Ann\Tor([\X]_n,\Y)\supseteq\Ann\Tor([\X]_{n-1},\Y)\cdot\Ann\Tor(\X,\Y)
$$
holds for any $n\ge1$.
Take elements $a\in\Ann\Tor([\X]_{n-1},\Y)$ and $b\in\Ann\Tor(\X,\Y)$.
Fix $i>0$, $Z\in[\X]_n$ and $Y\in\Y$.
Then there exists an exact sequence $0 \to L \to M \to N \to 0$ with $L\in[\X]_{n-1}$ and $N\in[\X]$ such that $Z$ is a direct summand of $M$.
This induces an exact sequence $\Tor_i^R(L,Y) \to \Tor_i^R(M,Y) \to \Tor_i^R(N,Y)$.
Since $a$ and $b$ annihilate $\Tor_i^R(L,Y)$ and $\Tor_i^R(N,Y)$ respectively, the element $ab$ annihilates $\Tor_i^R(M,Y)$.
\end{proof}

We state restriction made by finiteness of dimension.
The next result says that finite dimensional subcategories of $\CM(R)$ and $\lCM(R)$ containing the Cohen-Macaulay modules that are locally free on $\Spec_0(R)$ define the biggest nonfree loci and stable supports.

\begin{prop}\label{4.5}
Let $R$ be a Cohen-Macaulay local ring.
\begin{enumerate}[\rm(1)]
\item
Let $\X$ be a subcategory of $\CM(R)$ containing $\CM_0(R)$.
If $\X$ has finite dimension, then $\NF(\X)=\Sing R$.
\item
Suppose that $R$ is Gorenstein.
Let $\X$ be a subcategory of $\lCM(R)$ containing $\lCM_0(R)$.
If $\X$ has finite dimension, then $\lSupp\X=\Sing R$.
\end{enumerate}
\end{prop}

\begin{proof}
(1) Setting $\dim\X=n$, we find a module $G\in\X$ with $\X=[G]_{n+1}$.
We have
$$
\NF(\X)\overset{\rm(a)}{=}\NF(G)\overset{\rm(b)}{=}\V(\Ann\Ext(G,\mod R))\overset{\rm(c)}{=}\V(\Ann\Ext(\X,\mod R))\overset{\rm(d)}{=}\Sing R,
$$
where each equality follows from the following observation:\\
(a) Let $\p$ be a prime ideal not in $\NF(G)$.
Then $G_\p$ is $R_\p$-free.
For each $M\in\X$, we have $M_\p\in[G_\p]_{n+1}$, which implies that $M_\p$ is $R_\p$-free.
Hence $\p$ is not in $\NF(\X)$.\\
(b) This is obtained by Proposition \ref{4.4}.\\
(c) This follows from Lemma \ref{hos}.\\
(d) We have $\V(\Ann\Ext(\X,\mod R))=\NF(\X)\subseteq\Sing R$ by (a), (b) and (c).
Proposition \ref{5.1} shows that $\Sing R$ is contained in $\V(\Ann\Ext(\CM_0(R),\CM_0(R)))$, which is contained in $\V(\Ann\Ext(\X,\mod R))$ since $\CM_0(R)\subseteq\X\subseteq\mod R$.

(2) This statement is shown by an analogous argument to (1).
\end{proof}

Here, recall that a local ring $(R,\m)$ is called a {\em hypersurface} if the $\m$-adic completion of $R$ is a residue ring of a complete regular local ring by a principal ideal.
Also, recall that a Cohen-Macaulay local ring $R$ is said to have {\em minimal multiplicity} if the equality $\e(R)=\edim R-\dim R+1$ holds, where $\e(R)$ denotes the multiplicity of $R$ and $\edim R$ denotes the embedding dimension of $R$.

By the preceding proposition, finiteness of the dimension of a subcategory $\X$ implies
$$
\NF(\X)=\Sing R.
$$
Now we are interested in how this condition is close to the condition that $\X=\CM(R)$.
In fact, it turns out by some results in \cite{stcm,crs} that these two conditions are equivalent under certain assumptions.
We state this here, and by combining it with a result obtained in the previous section we give a criterion for a resolving subcategory to coincide with $\CM(R)$ in terms of the support and annihilator of $\Ext$.

\begin{prop}\label{3cond}
Let $(R,\m,k)$ be a Cohen-Macaulay local ring of dimension $d$.
Let $\X$ be a resolving subcategory of $\mod R$ contained in $\CM(R)$.
Suppose that one of the following three conditions is satisfied.
\begin{itemize}
\item
$R$ is a hypersurface.
\item
$R$ is locally a hypersurface on $\Spec_0(R)$, and $\X$ contains $\syz^dk$.
\item
$R$ is locally with minimal multiplicity on $\Spec_0(R)$, and $\X$ contains $\syz^dk$.
\item
$R$ is excellent and locally of finite Cohen-Macaulay representation type on $\Spec_0(R)$, and $\X$ contains $\syz^dk$ and a dualizing $R$-module.
\end{itemize}
Then the following statements hold.
\begin{enumerate}[\rm(1)]
\item
One has $\X=\NF_{\CM}^{-1}(\NF(\X))$.
Hence, $\X=\CM(R)$ if and only if $\NF(\X)=\Sing R$.
\item
Assume that $R$ is complete and equicharacteristic and that $k$ is perfect.
Then $\X=\CM(R)$ if and only if $\Supp\Ext(\X,\Y)=\V(\Ann\Ext(\X,\Y))$ for some subcategory $\Y$ of $\mod R$ containing $\CM_0(R)$.
\end{enumerate}
\end{prop}

\begin{proof}
(1) The former assertion follows from \cite[Main Theorem]{stcm} and \cite[Theorem 5.6 and Corollary 6.12]{crs}.
The latter assertion is shown by the former.

(2) We have $\Supp\Ext(\CM(R),\mod R)\subseteq\V(\Ann\Ext(\CM(R),\mod R))=\Sing R$, where the equality follows from Proposition \ref{cep}.
Let $\p\in\Sing R$, and set $M=\syz^d(R/\p)$.
Then $M$ is Cohen-Macaulay, and $\p$ belongs to $\NF(M)$.
We have $\NF(M)=\Supp\Ext_R^1(M,\syz M)$ (cf. \cite[Proposition 2.10]{res}), which is contained in $\Supp\Ext(\CM(R),\mod R)$.
Thus we have $\Supp\Ext(\CM(R),\mod R)=\V(\Ann\Ext(\CM(R),\mod R))$, which shows the `only if' part.

Suppose that $\Supp\Ext(\X,\Y)=\V(\Ann\Ext(\X,\Y))$ for some subcategory $\Y$ of $\mod R$ containing $\CM_0(R)$.
It is evident that the inclusions $\Supp\Ext(\X,\Y)\subseteq\NF(\X)\subseteq\Sing R$ hold.
Proposition \ref{cep} yields the equality $\Sing R=\V(\Ann\Ext(\X,\Y))$.
Hence we have $\NF(\X)=\Sing R$.
It follows from (1) that $\X=\CM(R)$.
Thus the `if' part is proved.
\end{proof}

\begin{rem}
As to the equivalence in Proposition \ref{3cond}(1), the `only if' part always holds, but `if' part does not hold without the assumption on the punctured spectrum.
Let
$$
R=k[[x,y,z]]/(x^2,y^2z^2),
$$
where $k$ is a field.
This is a $1$-dimensional complete intersection local ring with $\Spec R=\Sing R=\{\p,\q,\m\}$, where $\p=(x,y),\q=(x,z),\m=(x,y,z)$.
Let
$$
\X=\res_R(\m\oplus R/(x)).
$$
(For an $R$-module $M$ we denote by $\res_RM$ the {\em resolving closure} of $M$, i.e., the smallest resolving subcategory of $\mod R$ containing $M$.)
Then $\X$ is a resolving subcategory of $\mod R$ contained in $\CM(R)$ and containing $\syz^1k=\m$.
We have $\NF(\X)=\NF(\m\oplus R/(x))=\Sing R$ by \cite[Corollary 3.6]{res}.
However, $\X$ does not coincide with $\CM(R)$; the Cohen-Macaulay $R$-module $R/\p$ does not belong to $\X$.
Indeed, assume $R/\p\in\X$.
Then $\kappa(\p)$ belongs to $\res_{R_\p}((\m\oplus R/(x))_\p)=\res_{R_\p}(R_\p/xR_\p)$ by \cite[Proposition 3.5]{res}.
Every object of $\res_{R_\p}(R_\p/xR_\p)$ is a Cohen-Macaulay $R_\p$-module whose complexity is at most that of $R_\p/xR_\p$ by \cite[Proposition 4.2.4]{Av}.
Since $R_\p\cong k[[x,y,z]]_{(x,y)}/(x^2,y^2)$, the complexity of the $R_\p$-module $R_\p/xR_\p$ is $1$.
This shows that $\kappa(\p)$ has complexity at most $1$, which cannot occur because $R_\p$ is not a hypersurface (cf. \cite[Remark 8.1.1(3)]{Av}).
Thus $R/\p$ is not in $\X$.
\end{rem}

\section{Dimensions of $\CM_0(R)$ and $\lCM_0(R)$}

In this section, we consider finiteness of the dimensions of $\CM_0(R)$ and $\lCM_0(R)$.
It will turn out that it is closely related to the condition that $R$ is an isolated singularity.
Let us begin with studying over a Cohen-Macaulay local ring $R$ the relationship between finiteness of the dimensions of subcategories of $\CM(R),\lCM(R)$ and the $\m$-primary property of the annihilator of $\Tor,\Ext,\lhom$ on them, where $\m$ is the maximal ideal of $R$.

\begin{prop}\label{ky}
Let $(R,\m,k)$ be a Cohen-Macaulay local ring of dimension $d$.
\begin{enumerate}[\rm(1)]
\item
\begin{enumerate}[\rm(a)]
\item
Let $\X$ be a subcategory of $\CM_0(R)$ with $\dim\X<\infty$.
Let $\Y$ be any subcategory of $\mod R$.
Then $\Ann\Tor(\X,\Y)$ and $\Ann\Ext(\X,\Y)$ are $\m$-primary.
\item
Suppose that $R$ is Gorenstein.
Let $\X$ be a subcategory of $\lCM_0(R)$ with $\dim\X<\infty$.
Let $\Y$ be any subcategory of $\lCM(R)$.
Then $\Ann\lhom(\X,\Y)$ is $\m$-primary.
\end{enumerate}
\item
\begin{enumerate}[\rm(a)]
\item
Let $\X$ be a resolving subcategory of $\mod R$ contained in $\CM(R)$ and containing $\syz^dk$.
Let $\Y$ be a subcategory of $\mod R$ containing $\X$.
If $\Ann\Ext(\X,\Y)$ is $\m$-primary, then $\dim\X<\infty$.
\item
Suppose that $R$ is Gorenstein.
Let $\X$ be a thick subcategory of $\lCM(R)$ containing $\syz^dk$.
Let $\Y$ be a subcategory of $\lCM(R)$ containing $\X$.
If $\Ann\lhom(\X,\Y)$ is $\m$-primary, then $\dim\X<\infty$.
\end{enumerate}
\end{enumerate}
\end{prop}

\begin{proof}
(1) We prove only the assertion on $\Tor$ in (a) because the other assertions are similarly shown.
Let $n=\dim\X$.
Then there exists a module $G\in\X$ such that $\X=[G]_{n+1}$.
Lemma \ref{hos} implies $\V(\Ann\Tor(\X,\Y))=\V(\Ann\Tor(G,\Y))$, which is contained in $\V(\Ann\Tor(G,\mod R))$.
On the other hand, one sees from Proposition \ref{4.4} that $\V(\Ann\Tor(G,\mod R))=\NF(G)$ holds, and $\NF(G)$ is contained in $\{\m\}$ since $G\in\CM_0(R)$.
Therefore we obtain $\V(\Ann\Tor(\X,\Y))\subseteq\{\m\}$, which shows that $\Ann\Tor(\X,\Y)$ is an $\m$-primary ideal of $R$.

(2) We prove only the statement (a) because (b) follows from (a) and Proposition \ref{35}(2).
For some integer $h>0$ the ideal $\m^h$ annihilates $\Ext(\X,\Y)$.
Take a parameter ideal $Q$ of $R$ contained in $\m^h$.
Fix a module $M$ in $\X$.
Then $\syz^jM$ is in $\X$ since $\X$ is resolving, and it is also in $\Y$.
Hence we have $Q\Ext_R^i(M,\syz^jM)=0$ for $1\le i,j\le d$.
Since $M$ is Cohen-Macaulay, we can apply \cite[Proposition 2.2]{stcm} to $M$, which shows that $M$ is isomorphic to a direct summand of $\syz^d(M/QM)$.
As the ring $R/Q$ is artinian, there exists an integer $r>0$ such that $\m^r(R/Q)=0$.
We have a filtration
$$
M/QM\supseteq\m(M/QM)\supseteq\m^2(M/QM)\cdots\supseteq\m^r(M/QM)=0
$$
of $R/Q$-submodules of $M/QM$.
Decompose this into exact sequences $0 \to \m^{i+1}(M/QM) \to \m^i(M/QM) \to k^{\oplus s_i} \to 0$, where $0\le i\le r-1$.
Taking the $d$-th syzygies, we obtain exact sequences
$$
0 \to \syz^d(\m^{i+1}(M/QM)) \to \syz^d(\m^i(M/QM))\oplus R^{\oplus t_i} \to (\syz^dk)^{\oplus s_i} \to 0.
$$
By induction on $i$, we observe that the $R$-module $\syz^d(M/QM)$ belongs to $[\syz^dk]_r$, and hence $M\in[\syz^dk]_r$.
Since $\X$ is resolving and contains $\syz^dk$, we have $\X=[\syz^dk]_r$.
(Note here that $r$ is independent of the choice of $M$.)
Therefore $\dim\X\le r-1<\infty$.
%
\end{proof}

Recall that $R$ is called an {\em isolated singularity} if the local ring $R_\p$ is regular for all $\p\in\Spec_0(R)$.
Recall also that the {\em annihilator} of an $R$-linear additive category $\C$ is defined as 
$
\bigcap_{M,N\in\C}\Ann_R\Hom_\C(M,N).
$
The following is the first main result of this paper, which is a characterization of the isolated singularity of $R$ in terms of the dimensions of $\CM_0(R)$ and $\lCM_0(R)$:

\begin{thm}\label{main}
Let $R$ be a Cohen-Macaulay local ring with maximal ideal $\m$.
\begin{enumerate}[\rm(1)]
\item
Set the following four conditions.
\begin{enumerate}[\rm(a)]
\item
The dimension of $\CM_0(R)$ is finite.
\item
The ideal $\Ann\Ext(\CM_0(R),\CM_0(R))$ is $\m$-primary.
\item
The ideal $\Ann\Tor(\CM_0(R),\CM_0(R))$ is $\m$-primary.
\item
The ring $R$ is an isolated singularity.
\end{enumerate}
Then, the implications ${\rm(a)}\Leftrightarrow{\rm(b)}\Rightarrow{\rm(c)}\Rightarrow{\rm(d)}$ hold.
The implication ${\rm(d)}\Rightarrow{\rm(a)}$ also holds if $R$ is complete, equicharacteristic and with perfect residue field.
\item
Suppose that $R$ is Gorenstein, and set the following three conditions.
\begin{enumerate}[\rm(a)]
\item
The dimension of the triangulated category $\lCM_0(R)$ is finite.
\item
The annihilator of the $R$-linear category $\lCM_0(R)$ is $\m$-primary.
\item
The ring $R$ is an isolated singularity.
\end{enumerate}
Then the implications ${\rm(a)}\Leftrightarrow{\rm(b)}\Rightarrow{\rm(c)}$ hold, and so does ${\rm(c)}\Rightarrow{\rm(a)}$ if $R$ is complete, equicharacteristic and with perfect residue field.
\end{enumerate}
\end{thm}

\begin{proof}
Combine Propositions  \ref{5.1}, \ref{cep} and \ref{ky}. 

\end{proof}

The celebrated Auslander-Huneke-Leuschke-Wiegand theorem states that every Cohen-Macaulay local ring $R$ of finite Cohen-Macaulay representation type is an isolated singularity.
This was proved by Auslander \cite[Theorem 10]{A} when $R$ is complete, by Leuschke and Wiegand \cite[Corollary 1.9]{LW} when $R$ is excellent, and by Huneke and Leuschke \cite[Corollary 2]{HL} in the general case.
Our Theorem \ref{main} not only deduces this result but also improves it as follows:

\begin{cor}
Let $R$ be a Cohen-Macaulay local ring.
Suppose that there are only finitely many isomorphism classes of indecomposable Cohen-Macaulay $R$-modules which are locally free on $\Spec_0(R)$.
Then $R$ is an isolated singularity, and hence $R$ has finite Cohen-Macaulay representation type.
\end{cor}

\begin{proof}
Let $M_1,\dots,M_n$ be the nonisomorphic indecomposable Cohen-Macaulay $R$-modules which are locally free on $\Spec_0(R)$.
Then $\CM_0(R)$ contains $M:=M_1\oplus\cdots\oplus M_n$.
Since $\CM_0(R)$ is resolving, it also contains $[M]$.
On the other hand, take $N\in\CM_0(R)$.
Then each indecomposable summand of $N$ also belongs to $\CM_0(R)$, so it is isomorphic to one of $M_1,\dots,M_n$.
Hence $N$ is in $\add M$.
Therefore we have $\CM_0(R)=\add M=[M]$.
This implies $\dim\CM_0(R)=0<\infty$, and the assertion follows from Theorem \ref{main}(1).
\end{proof}

Our Theorem \ref{main} also gives rise to finiteness of the dimensions of $\CM(R)$ and $\lCM(R)$ as a direct consequence, the latter of which is nothing but \cite[Corollary 5.3]{ddc}.

\begin{cor}\label{comp}
Let $R$ be a Cohen-Macaulay equicharacteristic complete local ring with perfect residue field.
Suppose that $R$ is an isolated singularity.
Then $\CM(R)$ has finite dimension.
If $R$ is Gorenstein, $\lCM(R)$ has finite dimension as a triangulated category.
\end{cor}

\begin{rem}\label{prfhos}
In Corollary \ref{comp}, one can replace the assumption that $R$ is complete with the weaker assumption that $R$ is excellent, using a similar argument to the proof of \cite[Theorem 5.8]{ddc}.
(It is proved in \cite[Theorem 5.8]{ddc} that the latter statement in Corollary \ref{comp} holds true even if $R$ is not complete but excellent.)

Indeed, since the completion $\widehat R$ of $R$ is still an isolated singularity, Corollary \ref{comp} implies that $\CM(\widehat R)$ has finite dimension.
Putting $n=\dim\CM(\widehat R)$, we have $\CM(\widehat R)=[C]_{n+1}$ for some Cohen-Macaulay $\widehat R$-module $C$.
It follows from \cite[Corollary 3.6]{stcm} that there exists a Cohen-Macaulay $R$-module $G$ such that $C$ is isomorphic to a direct summand of the completion $\widehat G$ of $G$, and hence we have $\CM(\widehat R)=[\widehat G]_{n+1}$.
Now we claim the following.

\begin{claim*}
Let $m>0$.
For any $N\in[\widehat G]_m$ there exists $M\in[G]_m$ such that $N$ is isomorphic to a direct summand of $\widehat M$.
\end{claim*}

This claim is shown by induction on $m$.
When $m=1$, the module $N$ is isomorphic to a direct summand of a direct sum $\bigoplus_{i=1}^{h}\syz^{l_i}\widehat{G}$ for some $l_i\ge0$, and we can take $M=\bigoplus_{i=1}^{h}\syz^{l_i}G$.
Assume $m\ge 2$.
There is an exact sequence $\sigma: 0 \to X \to Z \to Y \to 0$ with $X\in[\widehat G]_{m-1}$ and $Y\in[\widehat G]$ such that $N$ is a direct summand of $Z$.
The induction hypothesis implies that there exist $V\in[G]_{m-1}$ and $W\in[G]$ such that $X$ and $Y$ are isomorphic to direct summands of $\widehat{V}$ and $\widehat{W}$, respectively.
We have isomorphisms $\widehat V\cong X\oplus X'$ and $\widehat W\cong Y\oplus Y'$, and get an exact sequence $\sigma':0 \to \widehat V \to Z' \to \widehat W \to 0$, where $Z'=X'\oplus Y'\oplus Z$.
Regard $\sigma'$ as an element of $\Ext_{\widehat R}^1(\widehat W,\widehat V)$.
We have $\Ext_{\widehat R}^1(\widehat W,\widehat V)\cong\Ext_R^1(W,V)^{\widehat{\ }}\cong\Ext_R^1(W,V)$, where the latter isomorphism follows from the fact that $\Ext_R^1(W,V)$ has finite length as an $R$-module.
This gives an exact sequence $\tau: 0 \to V \to M \to W \to 0$ such that $\widehat\tau\cong\sigma'$ as $\widehat R$-complexes.
Since $\widehat M\cong Z'=X'\oplus Y'\oplus Z$, the module $N$ is isomorphic to a direct summand of $\widehat M$.
As $M\in [G]_m$, the claim follows.

Let $X\in\CM(R)$.
Then the completion $\widehat X$ is in $\CM(\widehat R)=[\widehat G]_{n+1}$.
The claim implies that there exists $M\in[G]_{n+1}$ such that $\widehat{X}$ is isomorphic to a direct summand of $\widehat{M}$.
Hence $\widehat X\in\add_{\widehat R}(\widehat M)$.
It is seen by \cite[Lemma 5.7]{ddc} that $X$ is in $\add_RM$, whence $X\in[G]_{n+1}$.
\end{rem}

\section{Dimensions of more general resolving and thick subcategories}

In the preceding section, we studied finiteness of the dimensions of the resolving subcategory $\CM_0(R)$ of $\mod R$ and the thick subcategory $\lCM_0(R)$ of $\lCM(R)$.
The aim of this section is to investigate finiteness of the dimensions of more general resolving subcategories of $\mod R$ and thick subcategories of $\lCM(R)$.
We start by the following theorem.

\begin{thm}\label{mg}
Let $(R,\m,k)$ be a Cohen-Macaulay local ring.
Let $W$ be a specialization-closed subset of $\Spec R$ contained in $\Sing R$.
\begin{enumerate}[\rm(1)]
\item
Consider the following three conditions.\\
{\rm(a)} $\dim\NF_{\CM}^{-1}(W)<\infty$.
{\rm(b)} $W=(\emptyset\text{ or }\Sing R)$.
{\rm(c)} $\NF_{\CM}^{-1}(W)=(\add R\text{ or }\CM(R))$.\\
Then ${\rm(a)}\Rightarrow{\rm(b)}\Rightarrow{\rm(c)}$ hold.
\item
Let $R$ be Gorenstein.
Consider the following three conditions.\\
{\rm(a)} $\dim\lSupp^{-1}(W)<\infty$.
{\rm(b)} $W=(\emptyset\text{ or }\Sing R)$.
{\rm(c)} $\lSupp^{-1}(W)=(0\text{ or }\lCM(R))$.\\
Then ${\rm(a)}\Rightarrow{\rm(b)}\Rightarrow{\rm(c)}$ hold.
The implication ${\rm(c)}\Rightarrow{\rm(a)}$ also holds if $R$ is excellent and equicharacteristic and $k$ is perfect.
\end{enumerate}
\end{thm}

\begin{proof}
(1) As to the implication ${\rm(a)}\Rightarrow{\rm(b)}$, we may assume $W\ne\emptyset$.
Then $W$ contains $\m$, as $W$ is specialization-closed.
Hence we have $\CM_0(R)=\NF_{\CM}^{-1}(\{\m\})\subseteq\NF_{\CM}^{-1}(W)\subseteq\CM(R)$.
Proposition \ref{4.5} shows $\NF(\NF_{\CM}^{-1}(W))=\Sing R$.
Since $\NF(\NF_{\CM}^{-1}(W))$ is contained in $W$, we have $W=\Sing R$.
The implication ${\rm(b)}\Rightarrow{\rm(c)}$ is trivial.

(2) The proof of the implications ${\rm(a)}\Rightarrow{\rm(b)}\Rightarrow{\rm(c)}$ is similar to the corresponding implications in (1).
The implication ${\rm(c)}\Rightarrow{\rm(a)}$ is obtained by \cite[Theorem 5.8]{ddc}.
\end{proof}

In the rest of this section, we give several applications of our Theorem \ref{mg}.
First, we generalize some implications in Theorem \ref{main}.

\begin{cor}
Let $R$ be a Cohen-Macaulay local ring, and let $n$ be a nonnegative integer.
\begin{enumerate}[\rm(1)]
\item
If $\dim\CM_n(R)<\infty$, then $\dim\Sing R\le n$.
\item
Suppose that $R$ is Gorenstein.
If $\dim\lCM_n(R)<\infty$, then $\dim\Sing R\le n$.
The converse also holds if $R$ is excellent, equicharacteristic and with perfect residue field.
\end{enumerate}
\end{cor}

\begin{proof}
Let $W$ be the set of prime ideals $\p\in\Sing R$ such that $\dim R/\p\le n$.
Then $W$ is a specialization-closed subset of $\Spec R$ contained in $\Sing R$.
Since $W$ contains $\m$, it is nonempty.

(1) Since $\CM_n(R)=\NF_{\CM}^{-1}(W)$, Theorem \ref{mg}(1) shows $W=\Sing R$, which implies the inequality $\dim\Sing R\le n$.

(2) As $\lCM_n(R)=\lSupp^{-1}(W)$, the assertion follows from Theorem \ref{mg}(2).
\end{proof}

The next application of Theorem \ref{mg} is our second main result of this paper, which provides many sufficient conditions for a subcategory to have {\em infinite} dimension.

\begin{thm}\label{mgc}
Let $(R,\m,k)$ be a $d$-dimensional Cohen-Macaulay local ring.
One has $\dim\X=\infty$ in each of the following cases:
\begin{enumerate}[\rm(1)]
\item
$R$ is locally a hypersurface on $\Spec_0(R)$.\\
$\X$ is a resolving subcategory of $\mod R$ with $\syz^dk\in\X\subsetneq\CM(R)$.
\item
$R$ is Gorenstein and locally a hypersurface on $\Spec_0(R)$.\\
$\X$ is a thick subcategory of $\lCM(R)$ with $\syz^dk\in\X\ne\lCM(R)$.
\item
$R$ is locally with minimal multiplicity on $\Spec_0(R)$.\\
$\X$ is a resolving subcategory of $\mod R$ with $\syz^dk\in\X\subsetneq\CM(R)$.
\item
$R$ is excellent, admits a canonical module $\omega$ and locally has finite Cohen-Macaulay representation type on $\Spec_0(R)$.\\
$\X$ is a resolving subcategory of $\mod R$ with $\{\omega,\syz^dk\}\subseteq\X\subsetneq\CM(R)$.
\item
$R$ is a hypersurface.\\
$\X$ is a resolving subcategory of $\mod R$ with $\add R\ne\X\subsetneq\CM(R)$.
\item
$R$ is a hypersurface.\\
$\X$ is a thick subcategory of $\lCM(R)$ with $\{0\}\ne\X\ne\lCM(R)$.
\end{enumerate}
\end{thm}

\begin{proof}
Note from the assumption on $\X$ that $R$ is nonregular in each of the cases (1)--(4).
By Proposition \ref{3cond}(1), we have $\add R\ne\X=\NF_{\CM}^{-1}(\NF(\X))$ in the cases (1),(3),(4),(5), and $\{0\}\ne\X=\lSupp^{-1}(\lSupp\X)$ in the cases (2),(6).
Theorem \ref{mg} completes the proof.
\end{proof}

Denote by $\Db(\mod R)$ the bounded derived category of $\mod R$, and by $\perf R$ the subcategory of perfect complexes (i.e., bounded complexes of projective modules).
Recently, Oppermann and \v{S}\'{t}ov\'{i}\v{c}ek \cite[Theorem 2]{OS} proved that every proper thick subcategories of $\Db(\mod R)$ containing $\perf R$ has {\em infinite} dimension.
In the case where $R$ is a hypersurface, we can refine this result as follows:

\begin{cor}
Let $R$ be a local hypersurface.
Let $\X$ be a thick subcategory of $\Db(\mod R)$ with $\perf R\subsetneq\X\subsetneq\Db(\mod R)$.
Then the Verdier quotient $\X/\perf R$ has infinite dimension, and in particular so does $\X$.
\end{cor}

\begin{proof}
Note that a thick subcategory of $\Db(\mod R)$ contains $\perf R$ if and only if it contains $R$.
The equivalence $\lCM(R)\cong\Db(\mod R)/\perf R$ of triangulated categories given by Buchweitz \cite[Theorem 4.4.1]{B} corresponds each thick subcategory of $\lCM(R)$ to a subcategory of $\Db(\mod R)/\perf R$ of the form $\X/\perf R$, where $\X$ is a thick subcategory of $\Db(\mod R)$ containing $\perf R$.
Thus Theorem \ref{mgc}(6) implies that $\X/\perf R$ has infinite dimension.
The last assertion is easy (cf. \cite[Lemma 3.4]{R} or \cite[Lemma 3.5]{ddc}).
\end{proof}

\section*{Acknowledgments}
The authors thank Luchezar Avramov, Craig Huneke, Osamu Iyama and Srikanth Iyengar for their valuable comments and useful suggestions.
This work was done during the visits of the second author to University of Kansas in May, July and August, 2011.
He is grateful for their kind hospitality.
Results in this paper were presented at the seminars in University of Kansas, Nagoya University, University of Nebraska-Lincoln, University of Missouri and University of Isfahan.
The authors thank the organizers of these seminars.

\end{document}